\documentclass[11pt]{amsart}
\usepackage{hyperref}

\usepackage{amsfonts}
\usepackage{amsmath}
\usepackage{amssymb}
\usepackage{amstext}
\usepackage{amsthm}
\usepackage{mathpple}

\numberwithin{equation}{section}
\numberwithin{enumi}{equation}

\theoremstyle{plain}
\newtheorem{theorem}{Theorem}[section]

\newtheorem{lemma}[theorem]{Lemma}
\newtheorem{proposition}[theorem]{Proposition}
\newtheorem{corollary}[theorem]{Corollary}

\newtheorem*{conjecture*}{Conjecture}
\newtheorem*{theorem*}{Theorem \thmain}
\newtheorem*{lemma*}{Lemma}

\newtheorem{theoremintro}{Theorem}

\theoremstyle{remark}
\newtheorem{remark}[theorem]{Remark}

\DeclareMathOperator{\End}{End}
\DeclareMathOperator{\aut}{Aut}
\DeclareMathOperator{\gal}{Gal}

\DeclareMathOperator{\gsp}{GSp}
\DeclareMathOperator{\gl}{GL}
\DeclareMathOperator{\su}{SU}
\DeclareMathOperator{\gu}{GU}
\DeclareMathOperator{\Lie}{Lie}
\DeclareMathOperator{\lef}{Lef}

\DeclareMathOperator{\mt}{MT}
\DeclareMathOperator{\mat}{Mat}
\DeclareMathOperator{\jac}{Jac}

\global\let\sp\undefined
\DeclareMathOperator{\sp}{Sp}

\def\isotropic{I}
\def\anisotropic{J}

\def\ad{{\rm ad}}

\def\res{{\rm\bf R}}
\def\der{{\rm der}}

\def\units{^\cross}

\def\an{{\rm an}}

\def\ELL{{\mathbb L}}
\def\idp{{\mathfrak p}}
\def\idq{{\mathfrak q}}
\def\ff{{\mathbb F}}
\def\cx{{\mathbb C}}
\def\nat{{\mathbb N}}
\def\rat{{\mathbb Q}}
\def\integ{{\mathbb Z}}
\def\calo{{\mathcal O}}

\def\tensor{\otimes}
\def\cross{\times}
\def\ra{\rightarrow}
\def\inject{\hookrightarrow}
\def\ndiv{{\nmid}}
\def\inv{^{-1}}
\def\twiddle{\sim}
\def\iso{\cong}
\newcommand{\til}[1]{{\widetilde{#1}}}
\newcommand{\rest}[1]{|_{#1}}

\newcommand{\abs}[1]{{\left|#1\right|}}
\newcommand{\st}[1]{\{#1\}}

\newcommand{\ang}[1]{{{\langle #1 \rangle}}}
\newcommand{\oneover}[1]{\frac{1}{#1}}

\newenvironment{alphabetize}{\begin{enumerate}
\def\theenumi{\alph{enumi}}
}{\end{enumerate}}

\newenvironment{romanize}{\begin{enumerate}
\def\theenumi{\roman{enumi}}
}{\end{enumerate}}

\begin{document}
\title{Split reductions of simple abelian varieties}
\author{Jeffrey D. Achter}
\email{j.achter@colostate.edu}
\address{Department of Mathematics\\Colorado State University\\Fort
 Collins, CO 80523-1874}
\urladdr{http://www.math.colostate.edu/$\sim$achter}
\thanks{The author was partially supported by NSA grant H98230-08-1-0051.}

\begin{abstract}
Consider an absolutely simple abelian variety $X$ over a number field
$K$.  We show that if the absolute endomorphism ring of $X$ is
commutative and satisfies certain parity conditions, then $X_\idp$
is absolutely simple for almost all primes $\idp$.  Conversely, if the
absolute endomorphism ring of $X$ is noncommutative, then $X_\idp$ is
reducible for $\idp$ in a set of positive 
density.
\end{abstract}

\maketitle

An absolutely simple abelian variety over a number field may or may
not have absolutely simple reduction almost everywhere.  On one hand, let $K =
\rat(\zeta_5)$, and let $X$ be the Jacobian of the hyperelliptic curve with affine model 
\[
t^2 = s(s-1)(s-1-\zeta_5)(s-1-\zeta_5-\zeta_5^2)(s-1-\zeta_5-\zeta_5^2-\zeta_5^3),
\]
considered as an abelian surface over $K$.  Then
$X$ is absolutely simple \cite[p.648]{igusa60} and has ordinary
reduction at 
a set of primes $\idp$ of density one \cite[Prop.\ 1.13]{ibukiyamakatsuraoort}; at such primes
$X_\idp$ is absolutely simple.

On the other hand, let $Y$ be the Jacobian of the
hyperelliptic curve with affine model
\[
t^2 = s^6 - 12s^5 + 9s^4 - 32s^3 + 3s^2 +18s + 3,
\]
considered as an abelian surface over $L = \rat(\sqrt{2},\sqrt{-3})$.
Then $Y$ is absolutely simple \cite[Thm.\ 6.1]{bending}, but $Y_\idq$ is reducible
for each prime $\idq$ of good reduction.  (The conclusions about the
simplicity of $X_\idp$ and the reducibility of $Y_\idq$ follow from
Tate's description \cite{tateendff} of the endomorphism rings of
abelian varieties over finite fields.)

Note that $\End_K(X) \tensor\rat$ is the cyclotomic field
$\rat(\zeta_5)$, while $\End_L(Y)\tensor\rat$ is an indefinite
quaternion algebra over $\rat$.  Murty and Patankar study
the splitting behavior of abelian varieties over number fields, 
and advance the following
conjecture:
\begin{conjecture*}\cite[Conj.\ 5.1]{murtypatankar} \label{conjmp}
Let $X/K$ be an absolutely simple abelian variety over a number field.
The set of primes of $K$ where $X$ splits has positive density if and
only if $\End_{\bar K}(X)$ is noncommutative.
\end{conjecture*}
(A similar question has been raised by Kowalski; see \cite[Rem.\ 3.9]{kowalskiweil}.)
The present paper proves this conjecture under certain parity and
signature conditions on $\End(X)$.

The first main result states that a member of a large class of abelian
varieties with commutative endomorphism ring has absolutely simple
reduction almost everywhere.  (Throughout this paper, ``almost
everywhere'' means for a set of primes of density one.)
\begin{theoremintro}
\label{thmaincommutative}
Let $X/K$ be an absolutely simple abelian variety over a number
field.  Suppose that either
\begin{romanize}
\item $\End_{\bar K}(X) \tensor\rat \iso F$  a totally real field,
  and $\dim X/[F:\rat]$ is odd; or
\item $\End_{\bar K}(X) \tensor\rat\iso E$ a totally imaginary field,
  and the action of $E$ on $X$ is not special.
\end{romanize}
Then for almost every prime $\idp$, $X_\idp$ is absolutely simple.
\end{theoremintro}
(The notion of ``not special'' is discussed in Section
\ref{subseccommutative}; it is satisfied if, for instance, $\dim X$ is prime.)
Conversely, the second main result shows that abelian varieties with
noncommutative endomorphism ring have split reduction at a set of primes of
positive density.

\begin{theoremintro}
\label{thmainnoncommutative}
Suppose $X/K$ is an absolutely simple abelian variety over a number
field, and that $\End_{\bar K}(X)$ is noncommutative.
\begin{romanize}
\item For $\idp$ in a set of positive density, $X_\idp$ is
  absolutely reducible.
\item Suppose $\End_{\bar K}(X)\tensor\rat$ is
  an indefinite quaternion algebra over a totally real field $F$, and
  that $\dim X/2[F:\rat]$ is odd.  For $\idp$ in a set of positive
  density, $X_\idp$ is geometrically isogenous to the self-product of
  an absolutely simple abelian variety. 
\end{romanize}
\end{theoremintro}
Moreover, there is a finite extension of $K$ such that the set of
primes $\idp$ in Theorem \ref{thmainnoncommutative} actually has density one. 

Special cases of these results are already known.  The  case of
Theorem \ref{thmaincommutative}(i) in which $\End(X) \iso \integ$ is
due to Chavdarov \cite[Cor.\ 6.10]{chavdarov}; see also the related
work of Chai and Oort \cite{chaioort01}.  An abelian surface over a
finite field with noncommutative endomorphism ring is absolutely
reducible \cite[p.261]{murtypatankar05}, and the special case of
Theorem \ref{thmainnoncommutative}(i) in which $F = \rat$ and $\dim X
= 2$ is apparently well-known \cite[Cor.\ 2]{adimoolam}.  More recently, Murty
and Patankar have shown that if $X$ is either an abelian variety of CM
type \cite[Thm.\ 3.1]{murtypatankar}, or a modular abelian variety
with commutative absolute endomorphism ring \cite[Thm.\
4.1]{murtypatankar}, then $X$ has simple reduction almost everywhere.

Ellenberg et al.\ \cite{ellenbergetal} have addressed a
related problem for families of abelian varieties over a number
field.  Specifically, they consider the relative Jacobian of a family
of hyperelliptic curves $y^2=f(x)(x-t)$ over $K[t]$, and show that for
all but finitely many specializations of $t$ the resulting abelian
variety is simple.

The proof of Theorem \ref{thmainnoncommutative} uses the fact that, if
$\End_{\bar K}(X)$ is noncommutative, then the Tate module
$T_\ell(X)$ is a direct sum of copies of the same representation of
$\gal(K)$.  This, in turn, follows from the fact \cite{milnelefschetz}
that the first homology of $X$, as a representation of the Lefschetz
group, is isotypic but not irreducible.

The proof of Theorem \ref{thmaincommutative} is more involved, and
uses the Chebotarev 
theorem and the observation that if the Frobenius at $\idp$ acts irreducibly
on the $\ell$-torsion for some $\ell$, then $X_\idp$ is simple.  
This approach was used by Chavdarov in \cite[Cor.\ 
6.10]{chavdarov} in the special case where the image of $\gal(K)$, acting on each
$X_\ell$, is the group of symplectic similitudes
$\gsp_{2g}(\integ/\ell)$.   We give a more detailed outline of this
strategy in the following
example, which gives a quick  
proof of a special (but typical) case of \cite[Thm.\
3.1]{murtypatankar}.  Let $E$ be a totally imaginary extension of
$\rat$ of degree $2g$.  Suppose that $X/K$ is an absolutely simple
$g$-dimensional abelian variety with $\End_K(X) = \End_{\bar K}(X) \iso
\calo_E$.  Further suppose that the CM type of $E$ is nondegenerate in
the sense of \cite{kubotacm}.  For each rational prime $\ell$ there
are Galois representations $\rho_{X/K,\integ_\ell}: \gal(K) \ra
\aut(T_\ell(X)) \iso \gl_{2g}(\integ_\ell)$ and $\rho_{X/K,\ell}:\gal(K)
\ra \aut(X_\ell) \iso \gl_{2g}(\integ/\ell)$.  There is a set of
rational primes $\ELL$ (containing all but finitely many primes)
such that if
$\ell_1, \cdots, \ell_r$ are distinct primes in $\ELL$, then the image of $\gal(K)$
under the product representation $\cross_{1 \le i \le r}\rho_{X/K,\ell_i}$ is $\cross_{1 \le i \le r}
(\calo_E\tensor\integ/\ell_i)\units$ (e.g., \cite{ribet80}).

Let $\ell$ be any prime at which $E$ is inert, so that
$(\calo_E\tensor \integ/\ell)\units \iso \ff_{\ell^{2g}}\units$.  Let
$\isotropic_\ell$ be the set of elements of $\ff_{\ell^{2g}}\units$ which are
members of some proper subfield of $\ff_{\ell^{2g}}$.  Note that if
$\ff_{\ell^{2g}}$ is considered as a vector space over
$\integ/\ell$, then elements of $\isotropic_\ell$ are precisely the 
elements of $\ff_{\ell^{2g}}\units$ which acts reducibly on $\ff_{\ell^{2g}}$.
 There exists a
constant $C < 1$ such that for all $\ell$ inert in $E$, we have 
$\abs{\isotropic_\ell}/\abs{\ff_{\ell^{2g}}\units} < C$.

Let $M(X/K)$ be the set of (finite)  primes of $K$ where $X$ has good
reduction, and let $R(X/K)$ be the set of primes $\idp$ of good reduction
for which $X_\idp$ is reducible.
Suppose $\idp\in M(X/K)$, and let
$\sigma_\idp \in \gal(K)$ be a Frobenius element at $\idp$.  Let $\ell$ be a
rational prime relatively prime to $\idp$.  The Frobenius
endomorphism of $X_\idp$ acts as
$\rho_{X/K,\integ_\ell}(\sigma_\idp)$ on $T_\ell(X_\idp) \iso T_\ell(X)$.  If $X_\idp$ is
not simple, then the $\gal(\kappa(\idp))$-module $T_\ell(X_\idp)$ is
reducible, and in particular $\rho_{X/K,\ell}(\sigma_\idp)$ acts reducibly
on $X_{\idp,\ell}:=  X_\idp[\ell](\overline{\kappa(\idp)})$.  Therefore, if there exists one prime
$\ell$ such that $\rho_{X/K,\ell}(\sigma_\idp)$ acts irreducibly on
$X_{\idp,\ell}$, then $X_\idp$ is simple.

So, let $\ell_1, \cdots, \ell_r$ be distinct primes in $\ELL$ at which $E$ is
inert.  Let $R(X/K;\ell_1, \cdots, \ell_r)$ $\subset M(X/K)$ be the set
of primes $\idp$ such that for each $1 \le i \le r$,
$\rho_{X/K,\ell_i}(\sigma_\idp) \in \isotropic_{\ell_i}$. Then $R(X/K) \subseteq
R(X/K;\ell_1, \cdots, \ell_r)$.  By the Chebotarev theorem, the
density of $R(X/K;\ell_1, \cdots, \ell_r)$ is $\prod_{i=1}^r
\abs{\isotropic_{\ell_i}}/\abs{\ff_{\ell_i^{2g}}\units} < C^r$.  Since $C<1$
and we may take an arbitrarily large set of rational primes inert in $E$, the
density of $R(X/K)$ is zero, and the density of its complement is
therefore one.

Generalizing this argument to other abelian varieties with commutative
absolute endomorphism ring requires calculating the image of the
Galois representations $\rho_{X/K,\integ_\ell}$, which is
conjecturally described by the Mumford-Tate conjecture (see
Section \ref{subseclefschetz}); showing that a positive proportion
of elements of $\rho_{X/K,\ell}(\gal(K))$ act irreducibly on the Tate
module (Section \ref{subsecgroup}); and axiomatizing the foregoing
argument (Section \ref{subsecabvar}).

Quite recently, Banaszak et al.\ have extended the methods of
\cite{banaszaketal06} to abelian varieties of type III, which allows
an extension of Theorem \ref{thindefquat} to the case of definite
quaternion algebras.  Also, Zywina points out that sieve methods
(e.g., those of \cite{zywinathesis}) can be used to make
the density one statements in Section \ref{subseccommutative} more
explicit.  I will explain both of these developments in detail elsewhere.

\section{Groups of Lie type}
\label{subsecgroup}

If $B$ is a finite $A$-algebra, let $\res_{B/A}$ denote
Weil's restriction of scalars functor.

Group schemes $G/\integ[1/\Delta]$ of the following forms arise as the
images of Galois representations considered here: 
\begin{enumerate}\def\theenumi{\Alph{enumi}}
\item
\label{caseimag} There exist a totally imaginary field $E$ with
  maximal totally real subfield $F$; an $\calo_E[1/\Delta]$-module
  $V$ which is free of rank $2r$ over $\calo_F[1/\Delta]$; and an
  $\calo_E[1/\Delta]$-Hermitian pairing $\ang{\cdot,\cdot}$ on $V$;
  such that $G$ is the Weil restriction $G = \res_{\calo_E[1/\Delta]/\integ[1/\Delta]} \gu(V,  \ang{\cdot,\cdot})$.    Let $Z = E$.

\addtocounter{enumi}{1}
\item\label{casereal}
 There exist a totally real field $F$; a free
  $\calo_F[1/\Delta]$-module $V$ of rank $2r$; and an
  $\calo_F[1/\Delta]$-linear symplectic pairing $\ang{\cdot,\cdot}$ on
  $V$;
  such that $G = \res_{\calo_F[1/\Delta]/\integ[1/\Delta]}
  \gsp(V,\ang{\cdot,\cdot})$.  Let $Z =F$.

\end{enumerate}
The center $ZG$ of $G$ satisfies $ZG(\integ[1/\Delta])
\iso \calo_Z[1/\Delta]\units$.    The adjoint form of $G$ is $G^\ad :=
G/ZG$.  For each $\ell$ inert in $Z$, let
$T_\ell^\an\subset G(\integ/\ell)$ be a maximally anisotropic maximal 
torus.   

In case \eqref{caseimag}, the derived group of $G$ is $G^\der =
  \res_{\calo_E[1/\Delta]/\integ[1/\Delta]} \su(V,
  \ang{\cdot,\cdot})$.
Note that $G(\integ/\ell) \iso \gu(V\tensor \calo_E/\ell, \ang{\cdot,\cdot})$,
and $G^\der(\integ/\ell)\iso \su(V\tensor \calo_E/\ell,
\ang{\cdot,\cdot})$.   In particular, if $\ell$ is a rational prime
inert in $E$, then $G^\der(\integ/\ell) \iso \su_r(\calo_E/\ell)$.
Moreover, if $r$ is odd, then $T^\an_\ell$ acts irreducibly on
$V\tensor\integ/\ell$; while if $r$ is even, then $T^\an_\ell$ stabilizes two
subspaces which are in duality with each other.

In case \eqref{casereal}, the derived
  group of $G$ is $G^\der  = \res_{\calo_F[1/\Delta]/\integ[1/\Delta]}
  \sp(V,\ang{\cdot,\cdot})$.
Note that $G(\integ/\ell) \iso \gsp(V\tensor \calo_F/\ell, \ang{\cdot,\cdot})$
  and $G^\der(\integ/\ell)\iso \sp(V\tensor \calo_F/\ell,
  \ang{\cdot,\cdot})$.  In particular, if $\ell$ is a rational prime
  inert in $F$, then $G^\der(\integ/\ell) \iso
  \sp_{2r}(\calo_F/\ell)$.    Moreover, $T^\an_\ell$ acts irreducibly
on $V\tensor\integ/\ell$.

Let $G$ be a group scheme over $\integ[1/\Delta]$.  For a rational
prime $\ell\ndiv \Delta$, let $\anisotropic_\ell(G)$ be the set of all
$x \in G(\integ/\ell)$ for which the connected component of the
centralizer of $x$ is a torus which is maximally anistropic.
Let
$\isotropic_\ell(G)$ be the complement $G(\integ/\ell) -
\anisotropic_\ell(G)$.  Let $\anisotropic_{\ell,m}(G)$ be the set of
$x$ such that $x^m \in \anisotropic_\ell(G)$, and let
$\isotropic_{\ell,m}(G)$ be its complement.  Each of these sets is 
stable under conjugation.  
For $G$ of type \eqref{caseimag} or \eqref{casereal}, $x \in
\anisotropic_\ell(G)$ if and only if $x$ is $G(\integ/\ell)$-conjugate
to a regular element of $T_\ell^\an$.

Say that an
abstract group $H_\ell$ is of type $G(\integ/\ell)$ if there are
inclusions $G^\der(\integ/\ell) \subseteq H_\ell \subseteq
G(\integ/\ell)$.  For such a group $H_\ell$, let $\isotropic_{\ell,m}(H_\ell) = H_\ell
\cap \isotropic_{\ell,m}(G)$, and let $\anisotropic_{\ell,m}(H_\ell) = H_\ell \cap
\anisotropic_{\ell,m}(G)$.

\begin{lemma}
\label{lemcexists}
Suppose $G/\integ[1/\Delta]$ is a group of type \eqref{caseimag},
or \eqref{casereal}, and let $m$ be a natural number.
There exists a constant $C = C(m,G)$ such that if $\ell$ is inert in
$Z$ and sufficiently large, and if $H_\ell$ is of type
$G(\integ/\ell)$, then $\abs{\isotropic_{\ell,m}(H_\ell)}/\abs{H_\ell}
< C$.
\end{lemma}

\begin{proof}
First, for 
each $m \in \nat$ we
show the existence of a positive constant $D_0(m,G)$ such that for all
sufficiently large $\ell$ inert in $Z$, 
$\abs{\anisotropic_{\ell,m}(G)}/ \abs{G(\integ/\ell)} > D_0(m,G)$.  Subsequently,
we show how to deduce a uniform statement for all $H_\ell$ of type
$G(\integ/\ell)$.

Let $T^*_{\ell}$ be the set of regular elements of $T^\an_\ell$, and
let $T^*_{\ell,m}$ be the set of $x \in T^\an_\ell$ such that $x^m \in
T^*_\ell$.  There are monic polynomials $f$ and $f^*$ of the same
degree such that $\abs{T^\an_\ell} = f(\ell)$ and $\abs{T^*_\ell}=
f^*(\ell)$ \cite{fleischmannjaniszczak93}.  (In fact,
\cite{fleischmannjaniszczak93} works out the analogous polynomials for
$\abs{T^*_\ell \cap G'(\integ/\ell)}$, but the result for
$G(\integ/\ell)$ itself follows immediately.)  Therefore, there exists
a constant $B$ such that, if $\ell \gg 0$, then
$\abs{T^*_\ell}/\abs{T^\an_\ell} > 1 - B/\ell$.  By considering the
fibers of the $m^{th}$ power map, we see that
$\abs{T^*_{\ell,m}}/\abs{T^\an_\ell} > 1 - mB/\ell$.

An element of $G(\integ/\ell)$ is in $\anisotropic_{\ell,m}(G)$ if and
only if it is conjugate to an element of $T_{\ell,m}^*$.  
The normalizer $N_\ell = N_{G(\integ/\ell)}(T^\an_\ell)$ is an extension of a
  finite group $W$ by $T^\an_\ell$; the group $W$ depends on
  $G$, but not on $\ell$.  Moreover, $T^*_{\ell,m}$ is stable under the
  action of $N_\ell$.  We obtain the estimate
\begin{align*}
\frac{\abs{\anisotropic_{\ell,m}(G)}}{\abs{G(\integ/\ell)}} &= 
\oneover{\abs{G(\integ/\ell)}}\left(\frac{\abs{G(\integ/\ell)}}{\abs{N_\ell}}
\abs{T_{\ell,m}^*}
\right) \\
&= \oneover{\abs
  W}\frac{\abs{T_{\ell,m}^*}}{\abs{T^\an_\ell}} 
 > \oneover{\abs W}\left(1 - \frac{mB}{\ell}\right).
\end{align*}
This shows the existence of $D_0(m,G)$ with the desired properties.
Membership in 
$\anisotropic_{\ell,m}(G)$ is well-defined on cosets modulo the center of $G$.
Therefore, the proportion of elements of $G^\ad(\integ/\ell)$ which
are (represented by) elements whose 
$m^{th}$ power is maximally anisotropic is also at least $D_0(m,G)$.

Now let $H_\ell$ be any group of type $G(\integ/\ell)$, and let
$H_\ell^\ad  = H_\ell/(ZG(\integ/\ell)\cap H_\ell)$.  There is an
inclusion of groups $H_\ell^\ad \inject 
G^\ad(\integ/\ell)$, with cokernel a finite cyclic group whose order
$n$ divides the rank of $G$.

Suppose $x \in \anisotropic_{\ell,mn}(G)$.  The equivalence class of $x^n$ modulo
the center is represented by an element $h$ of $H_\ell$.  Moreover,
for such an $h$, $h^m \equiv x^{mn} \bmod ZG(\integ/\ell)$ is
maximally anisotropic modulo the center.  The elements of $G^\ad(\integ/\ell)$
which are maximally anisotropic give rise to at least $\oneover n D_0(mn,
G)\abs{G^\ad(\integ/\ell)}$ distinct maximally anisotropic elements of
$H^\ad_\ell$.  Let $D(m,G) = \min\st{\oneover n D_0(mn,G) : n|\text{rank}(G)}$.  Then one may take $1-D(m,G)$ for $C(m,G)$ in the statement of
Lemma \ref{lemcexists}.
\end{proof}

\begin{remark}
For groups of type \eqref{casereal}, the case $m=1$ and $H_\ell =
G(\integ/\ell)$ of Lemma
\ref{lemcexists} is proved
in \cite[Cor.\ 3.6]{chavdarov}.
\end{remark}

\begin{lemma}
\label{lemgoursat}
Let $G/\integ[1/\Delta]$ be a group.  Suppose that either $G$ is of
type \eqref{caseimag} with $r\ge 2$ or that $G$ is of type \eqref{casereal}.
Let $\ell_1, \cdots, \ell_m$ be distinct rational primes which are
inert in $Z$.  Let $H$ be a subgroup of $G^\der(\integ/(\prod
\ell_i))$ such that for each $i$, the 
composition  $H \inject G^\der(\integ/(\prod \ell_i)) \ra
G^\der(\integ/\ell_i)$ is surjective.  Then $H = G^\der(\integ/(\prod
\ell_i))$.
\end{lemma}

\begin{proof}
This is Goursat's lemma \cite[p.\ 793]{ribet76}; see also
\cite[Prop.\ 5.1]{chavdarov}.  The hypothesis guarantees that the
adjoint groups $G^\ad(\integ/\ell_i)$ are distinct
nonabelian simple groups.
\end{proof}

\begin{lemma}
\label{lemgptheory}
Let $r\in \nat$ and let $\ff$ be a finite field, with $\ff \not \in
\st{\ff_2, \ff_4, \ff_3,\ff_9}$.  Suppose that either $G$ is
$\gu_r/\ff$ and $r \ge 2$ or that $G$ is $\gsp_{2r}/\ff$.  Let
$G^\der$ be the derived group of $G$, let $G^\ad$ be the adjoint form
of $G$, and let $\alpha:G \ra  G^\ad$ and $\beta:G^\der \ra G^\ad$ be
the canonical projections.  Let $H\subset G(\ff)$ be a subgroup.  If
$\beta\inv(\alpha(H)) = G^\der(\ff)$, then $H$
contains $G^\der(\ff)$.
\end{lemma}

\begin{proof}
This is standard; the hypothesis on $\ff$ rules out exceptional cases.
\end{proof}

\section{Abelian varieties}
\label{subsecabvar}

Let $X/k$ be a principally polarized abelian variety over a field $k$.  For each rational
prime $\ell$ invertible in $k$, let $T_\ell(X)$ be the $\ell$-adic
Tate module of $X$, and let $X_\ell:= X[\ell](\bar k) = T_\ell(X)/
\ell T_\ell(X)$. Then $T_\ell X$ and $X_\ell$ come equipped with an action by $\gal(k)$.  Let
$\rho_{X/k,\integ_\ell}: \gal(k) \ra \aut(T_\ell(X))$ and
$\rho_{X/k,\ell}: \gal(k) \ra \aut(X_\ell)$ be the associated
representations, with respective images $H_{X/k,\integ_\ell}$ and 
$H_{X/k,\ell}$.  Let $H_{X/k,\rat_\ell}$ be the Zariski closure of
$H_{X/k,\integ_\ell}$ in $\aut(T_\ell(X) \tensor \rat)$. 

Suppose $X$ is simple.  Then the endomorphism 
algebra $D(X) = \End(X) \tensor\rat$ is a central simple algebra over
a number field $E(X)$ with positive involution.  Let $F(X)\subseteq E(X)$ be the subfield
fixed by the involution.  Then $F(X)$ is a totally real field, and either
$E(X) = F(X)$ or $E(X)$ is a totally imaginary quadratic extension of
$F(X)$.  Let $f(X) = [F(X):\rat]$, let $e(X) = [E(X):\rat]$, and let $d(X) =
\sqrt{[D(X):E(X)]}$.   

If $X$ and $Y$ are isogenous abelian varieties, write
$X\twiddle Y$.

If $K$ is a number field, let $M_K$ be the set of (finite) primes of
$K$; if $\idp \in M_K$, denote its residue field by $\kappa(\idp)$.
Suppose $X/K$ is an absolutely simple abelian variety.  As in the
introduction, let $M(X/K)\subset M_K$ be the set of primes of good reduction of
$X$.  It is convenient to distinguish the following subsets of
$M(X/K)$:
\begin{itemize}
\item $S(X/K) = \st{ \idp \in M(X/K) : X_\idp\text{ is simple}}$;
\item $S^*(X/K) = \st{ \idp \in M(X/K): X_\idp\text{ is absolutely
      simple}}$;
\item $R(X/K) = \st{\idp \in M(X/K): X_\idp\text{ is reducible}}$;
\item $R^*(X/K) = \st{\idp \in M(X/K): X_\idp\text{ is absolutely
      reducible}}$.
\end{itemize}

Then $R(X/K)$ is the complement of $S(X/K)$; $R^*(X/K)$ is the
complement of $S^*(X/K)$; $S^*(X/K) \subset S(X/K)$; and $R^*(X/K)
\supset R(X/K)$.  In this notation, \cite[Conj.\ 5.1]{murtypatankar}
states that $S(X/K)$ has density one if and only if $\End_{\bar K}(A)$
is commutative.

Many attributes of $X_{\bar K}$, the base change of $X$ to an
algebraic closure of $K$, are already detectable over a finite
extension of $K$.  Consider the following condition on an
abelian variety $X$ over a number field $K$ and a finite, Galois
extension $K'/K$:

\begin{equation}
\label{eqfieldcond}
\parbox{4in}{ $\End_{K'}(X) = \End_{\bar K}(X)$; for all but finitely
  many $\idp \in M(X/K)$, if $\idp' \in M(X/K')$ is a prime which divides $\idp$, and if
  $X_{\idp'}$ is simple, then $X_\idp$ is absolutely simple; and 
  $H_{X/K',\rat_\ell}$ is connected for each rational prime $\ell$.
}
\end{equation}

\begin{lemma}
\label{lemgoodfield}
Let $X/K$ be an abelian variety over a number field.  Fix a natural
number $n \ge 5$, and let $K'/K$ be a finite, Galois extension which
contains the field of definition of all $n$-torsion 
points of $X$.   Then $(X/K,K')$ satisfies \eqref{eqfieldcond}.
\end{lemma}

\begin{proof}
There are three conditions in \eqref{eqfieldcond}.  The first follows from
Silverberg's criterion \cite[Thm.\ 
2.4]{silverberg92}.  The second follows from this and the fact that,
if $\idp$ is relatively prime to $n$, then
$X[n](K') \inject X_{\idp'}[n](\kappa(\idp'))$.    The final condition
is
\cite[Thm.\ 4.6]{silverbergzarhinconnected}.
\end{proof}

In Lemma \ref{lemgoodfield}, if one insists that $n$ be divisible by
two distinct primes $n_1$ and $n_2$, each of which is at least five,
then the second condition of \eqref{eqfieldcond} holds for all primes $\idp
\in M(X/K)$.

Throughout this paper Lemma \ref{lemgoodfield} will be used, often
implicitly, to show the existence of an extension $K'/K$ such that
$(X/K,K')$ satisfies \eqref{eqfieldcond}.

We will often work with an abelian variety $X/K$, a group scheme
$G/\integ[1/\Delta]$, and an infinite set of rational primes $\ELL \subset
M_\rat$ relatively prime to $\Delta$ which satisfy the following hypotheses:
\begin{equation}
\label{eqgoursat}
\parbox{4in}{The abelian variety $X$ is absolutely simple.
For each $\ell \in M_\rat$, $H_{X/K,\ell}$ is isomorphic to
a subgroup of $G(\integ/\ell)$.    For each $\ell \in \ELL$,
$H_{X/K,\ell}$ is of type $G(\integ/\ell)$.
For each finite subset $A \subset
\ELL$, the image of $\gal(K)$ under $\cross_{\ell \in A}
\rho_{X,\ell}$ is $\cross_{\ell \in A}H_{X/K,\ell}$.
}
\end{equation}

If $(X/K, G/\integ[1/\Delta], \ELL)$ satisfies \eqref{eqgoursat}, and if $A \subset
M_\rat$ is any set of primes, let $\isotropic(X/K; G; A) \subset M(X/K)$
be the set of primes $\idp$ such that for each $\ell \in
A$ and each Frobenius element $\sigma_\idp \in \gal(K)$ at $\idp$,
$\rho_{X,\ell}(\sigma_\idp) \in \isotropic_\ell(G)$.  Its complement $\anisotropic(X/K;G;A)$ is
the set of primes $\idp \in M(X/K)$ for which there exists some prime
$\ell \in A$ such that $\rho_{X,\ell}(\sigma_\idp) \in \anisotropic_\ell(G)$.  

Let $\isotropic(X/K; G) = \isotropic(X/K; G; M_\rat)$ and let $\anisotropic(X/K;G) = \anisotropic(X/K; G; M_\rat)$.
Note that for any $A\subset M_\rat$, $\isotropic(X/K; G) \subseteq \isotropic(X/K; G;
A)$ and $\anisotropic(X/K;G;A) \subseteq \anisotropic(X/K;G)$.

\begin{lemma}
\label{lemdensityanisoone}
Suppose $(X/K, G/\integ[1/\Delta], \ELL)$ satisfies \eqref{eqgoursat}.  Suppose
that there is a constant $C<1$ such that for each $\ell \in \ELL$ and
each group $H_\ell$ of type $G(\integ/\ell)$,
$\abs{\isotropic_\ell(H_\ell)}/\abs{H_\ell} < C$.  Then $\anisotropic(X/K;
G)$ has density one.
\end{lemma} 

\begin{proof}
Let $A\subset \ELL$ be a finite subset.  The
Chebotarev density theorem, applied to the representation
$\cross_{\ell \in A}\rho_{X/K,\ell}$ of $\gal(K)$, shows that the
density of $\isotropic(X/K; G; A)$ is $\prod_{\ell \in
  A}\abs{\isotropic_\ell(H_{X/K,\ell})}/\abs{H_{X/K,\ell}}$, which is less than $C^{\abs
A}$. By taking $A$ arbitrarily large, we find that $\isotropic(X/K;G)$ has
density zero and its complement, $\anisotropic(X/K;G)$, has density one.
\end{proof}

Recall that if $G$ is of type \eqref{casereal}, or of type
\eqref{caseimag} with $r$ odd, and if $\ell$ is inert in $Z$, then the
natural representation of $G(\integ/\ell)$ is an irreducible module
over $T^\an_\ell$.  Equivalently, some semisimple element of $G(\integ/\ell)$
acts irreducibly on $V\tensor\integ/\ell$.

\begin{lemma}
\label{lemfrobirred}
Suppose $(X/K, G/\integ[1/\Delta], \ELL)$ satisfies \eqref{eqgoursat}.  
Suppose $\idp\in M(X/K)$, and let
$\sigma_\idp$ be a Frobenius element at $\idp$.  If
$\rho_{X/K,\ell}(\sigma_\idp) 
\in \anisotropic_\ell(G)$, and if some semisimple element of $G(\integ/\ell)$ acts
irreducibly on $X_\ell$,  then the reduction $X_\idp$ is
simple.
\end{lemma}

\begin{proof}
Let $\ell$ be a
rational prime relatively prime to $\idp$, and suppose
$\rho_{X/K,\ell}(\sigma_\idp) \in \anisotropic_\ell(G)$.   Then the group
$\ang{\rho_{X/K,\ell}(\sigma_\idp)}$ acts irreducibly on $X_\ell$, so
that $\ang{\rho_{X/K,\integ_\ell}(\sigma_\idp)}$ acts 
  irreducibly on $T_\ell(X)$.   The Tate module of $X$ is an
  irreducible $\gal(\kappa(\idp))$-module, thus the abelian variety
  $X_\idp/\kappa(\idp)$ is simple  
\cite[Thm.\ 1(b)]{tateendff}. 
\end{proof}

\begin{lemma}
\label{lemfrobirredunitary}
Suppose $(X/K,G/\integ[1/\Delta],\ELL)$ satisfies \eqref{eqgoursat},
and that $G$ is of type \eqref{caseimag} with $r$ even.  Suppose $\ell
\in \ELL$ is inert in $Z$, and that $\idp \in M(X/K)$.  If
$\sigma_\idp$ is a Frobenius element at $\idp$, if
$\rho_{X/K,\ell}(\sigma_\idp) \in \anisotropic_\ell(G)$, and if
$X_\ell$ is the natural representation of $G(\integ/\ell)$, then the
reduction $X_\idp$ is simple.
\end{lemma}

\begin{proof}
Possibly after conjugating, assume that $t:=
\rho_{X/K,\ell}(\sigma_p)$ lies in $T_\ell^\an$.
Recall that $T^\an_\ell$ stabilizes two maximal isotropic
subspaces $W_1$ and $W_2$ of $X_\ell$ which are in duality with each
other; the action of $T^\an_\ell$ on $W_2$ is the Frobenius
twist of its action on $W_1$.  By Tate's theorem, either $X_\idp$ is
irreducible, or $X_\idp \twiddle Y_1 \cross Y_2$ with each $Y_j$
irreducible.  In the latter case, the polarization would place $Y_1$
and $Y_2$ in duality, and in particular $Y_1$ and $Y_2$ are
isogenous.  However, since $t$ is a regular element of 
$T^\an_\ell$, its eigenvalues on $W_1$ are distinct from its
eigenvalues on $W_2$.  Therefore, $X_\idp$ is
irreducible.
\end{proof}

\begin{lemma}
\label{lemdensitySone}
Suppose $(X/K, G/\integ[1/\Delta], \ELL)$ satisfies \eqref{eqgoursat}.  Suppose
that there is a constant $C<1$ such that for each $\ell \in \ELL$ and
each group $H_\ell$ of type $G(\integ/\ell)$, 
$\abs{\isotropic_\ell(H_\ell)}/\abs{H_\ell} < C$.  Suppose that for
each $\ell \in \ELL$, either
\begin{alphabetize}
\item some semisimple element of $G(\integ/\ell)$
  acts irreducibly on $X_\ell$; or
\item $G$ is of type \eqref{caseimag}, $r$ is even, $\ell$ is inert in
  $Z$, and $X_\ell$ is the
natural representation.
\end{alphabetize}
Then $S(X/K)$ has density one.
\end{lemma}

\begin{proof}
Part (a) follows immediately from Lemmas \ref{lemdensityanisoone} and
\ref{lemfrobirred}.  Part (b) follows from Lemmas
\ref{lemdensityanisoone} and \ref{lemfrobirredunitary}.
\end{proof}

In the other direction, we have:

\begin{lemma}
\label{lemnewRall}
Suppose $X/K$ is an absolutely simple abelian variety over a number
field, and suppose $\idp\in M(X/K)$.  Suppose that there exist a prime
$\ell$ relatively prime to $\idp$, an integer $d\ge 2$,  and a $\rat_\ell$-representation
$W_{\rat_\ell}$ of $\gal(K)$  with $T_\ell(X)\tensor\rat \iso
W_{\rat_\ell}^{\oplus d}$ as $\gal(K)$-module. 
\begin{alphabetize}
\item There are simple abelian varietes $Y_1, \ldots, Y_s$ over
  $\kappa(\idp)$ such that
\begin{equation}
\label{eqisogdecomp}
X_\idp \twiddle Y_1^{e_1}  \cross \cdots \cross Y_s^{e_s}.
\end{equation}
For each $j$ with $1 \le j \le s$, $d | e_j \cdot d(Y_j)$.
\setcounter{enumi}{1}  
\item If the residue field $\kappa(\idp)$ is a field of prime order,
  then each $e_j \ge d$.  In particular, $X_\idp$ is not simple.
\end{alphabetize}
\end{lemma}

\begin{proof}
Recall that $f_\idp(t)$, the characteristic polynomial of Frobenius of
$X_\idp$, coincides with the characteristic polynomial of
$\rho_{X/K,\integ_\ell}(\sigma_\idp)$.  Since $T_\ell(X)\tensor \rat \iso
W_{\rat_\ell}^{\oplus d}$, there exists a polynomial $g_{\idp,\ell}(t)
\in \integ_\ell(t)$ with $f_\idp(t) = g_{\idp,\ell}(t)^d$.    Note that $f_\idp$
and $g_{\idp,\ell}$ are both monic.  By inductively analyzing the coefficients of
$g_{\idp,\ell}(t)$ (in descending order), one sees that $g_{\idp,\ell}(t) \in \rat[t]\cap
\integ_\ell[t] \subset \rat_\ell[t]$; by Gauss's lemma, $g_{\idp,\ell}(t)
\in \integ[t]$.  Factor $g_{\idp,\ell}(t) = g_1(t)^{a_1} \cdots g_s(t)^{a_s}$ as a
product of powers of distinct irreducible polynomials, so that
$f_\idp(t) = g_1(t)^{a_1d} \cdots g_s(t)^{a_sd}$.  Consider some
$j$ with $1 \le j \le s$.  By the theory developed by Tate and Honda  
\cite[Thm.\ 1(b) and Thm. 2(e)]{tateendff} \cite[Thm.\ 1 and Rem.\ 2]{tatehonda},
there is a simple abelian variety $Y_j$ over $\kappa(\idp)$ with
characteristic polynomial of Frobenius $g_j(t)^{d(Y_j)}$.  Moreover,
any abelian variety over $\kappa(\idp)$ with characteristic polynomial
divisible by $g_j(t)$ contains a sub-abelian variety isogenous to
$Y_j$.  From this the decomposition \eqref{eqisogdecomp} follows,
where $e_j = \frac{a_jd}{d(Y_j)}$.  This proves (a).

For (b), a simple abelian variety over a prime field has commutative
endomorphism ring.  (This follows from \cite[Thm.\ 1(ii)]{tatehonda},
and was noted in \cite[p.\ 469]{chiII}.)   Therefore, each $Y_j$ has
commutative endomorphism ring; $d(Y_j)= 1$; and each exponent $e_j$ in
\eqref{eqisogdecomp} is a multiple of $d \ge 2$.
\end{proof}

\begin{lemma}
\label{lemnewpowsimp}
Suppose $(X/K, G/\integ[1/\Delta], \ELL)$ satisfies \eqref{eqgoursat}, where
$X$ is an absolutely simple abelian variety.   Suppose there is a constant $C<1$ such that
for each $\ell \in \ELL$
and each group $H_\ell$ of type $G(\integ/\ell)$, 
$\abs{\isotropic_\ell(H_\ell)}/\abs{H_\ell} < C$.  Suppose there is an integer $d
\ge 2$ such that for each $\ell \in \ELL$ there exists an irreducible
$\gal(K)$-module $W_{\rat_\ell}$ with $T_\ell(X)\tensor\rat \iso
W_{\rat_\ell}^{\oplus d}$.  Finally, suppose there exists $\sigma \in \gal(K)$
such that $\rho_{X/K,\integ_\ell}$ acts irreducibly and semisimply on
$W_{\rat_\ell}$. 
Then for $\idp$ in a set of density one there exists a simple abelian
variety $Y_\idp/\kappa(\idp)$ such that $X_\idp$ is isogenous to $Y_\idp^{\oplus d}$.
\end{lemma}

\begin{proof}
Suppose $\idp \in M(X/K)$, and choose $\ell \in \ELL$ prime to $\idp$.  Let $g_{\idp,\ell}(t) \in \integ[t]$ be
the characteristic polynomial of $\sigma_\idp$ acting on
$W_{\rat_\ell}$ via $\rho_{X/K,\integ_\ell}$, and let $f_{\idp}(t)$ be
the characteristic polynomial of Frobenius of $X_\idp$.  We have seen
(Lemma \ref{lemnewRall}) that $f_\idp(t) = g_{\idp,\ell}(t)^d$.  

By Lemma \ref{lemdensityanisoone}, $\anisotropic(X/K;G)$ has density
one.  If $\idp\in \anisotropic(X/K;G)$, then $\sigma_\idp$ acts
irreducibly on $W_{\rat_\ell}$, and thus $g_{\idp,\ell}(t)$ is
irreducible (over $\rat$).  Therefore, for such $\idp$, $s = 1$ in
\eqref{eqisogdecomp}, and $X_\idp \twiddle Y^e$ for some $e$ with
$d(Y) \cdot e = d$.

If we further restrict $\idp$ to have residue degree one (which is
still a density-one condition), then $d(Y) = 1$ and $e = d$ (Lemma
\ref{lemnewRall}(b)).
\end{proof}

In fact, we will need slightly stronger variants of Lemmas 
\ref{lemdensitySone} and \ref{lemnewRall}.

\begin{proposition}
\label{propdensitySstarone}
Let $X/K$ be an absolutely simple abelian variety over a number field,
and let $K'/K$ be a finite Galois extension of degree $m$ such that
$(X/K,K')$ satisfies \eqref{eqfieldcond}.   Suppose
$(X/K',G/\integ[1/\Delta], \ELL)$ satisfies \eqref{eqgoursat}.
Suppose that there is a
constant $C<1$ such that for all $\ell \in \ELL$ and each $H_\ell$ of
type $G(\integ/\ell)$, 
$\abs{\isotropic_{\ell,m}(H_\ell)}/\abs{H_\ell}< C$.  
Suppose that for each $\ell \in \ELL$, either
\begin{alphabetize}
\item some semisimple element of $G(\integ/\ell)$
  acts irreducibly on $X_\ell$; or
\item $G$ is of type \eqref{caseimag}, $r$ is even, $\ell$ is inert in
  $Z$, and $X_\ell$ is the
natural representation.
\end{alphabetize}
Then $S^*(X/K)$ has
density one.
\end{proposition}

\begin{proof}
  We indicate how to prove Lemmas \ref{lemdensityanisoone} and
  \ref{lemfrobirred} in this more general setting.  This will prove
  Proposition \ref{propdensitySstarone} under hypothesis (a); the
  result for hypothesis (b) is entirely analogous.  Let $B =
  \gal(K'/K)$, and for each $\ell$ let $B_\ell =
  H_{X/K,\ell}/H_{X/K',\ell}$.  Then $B_\ell$ is a quotient of $B$.

Let $\anisotropic_m(X/K; G)$ be the set of primes $\idp\in M(X/K)$ for which there
exists some $\ell\in \ELL$ such that $\rho_{X/K,\ell}(\sigma_\idp)^m
\in \anisotropic_\ell(G)$, i.e., such that $\rho_{X/K,\ell}(\sigma_\idp) \in
\anisotropic_{\ell,m}(H_{X/K,\ell})$.  We start by showing that $\anisotropic_m(X/K;G)$ has density
one. 

Suppose $A\subset \ELL$ is a finite set.  The hypothesis \eqref{eqgoursat},
applied to the subgroup $\prod_{\ell \in A} H_{X/K',\ell}$ of 
$\prod_{\ell \in A} H_{X/K,\ell}$, implies that there is a quotient
$B_A$ of $B$ such that the image of $\gal(K)$
under $\cross_{\ell \in A} \rho_{X/K,\ell}$ is an extension of $B_A$ by
$\prod_{\ell \in A} H_{X/K',\ell}$.

Suppose $\ell \in \ELL$.  Since
$\abs{B_\ell}$ is bounded independently of $\ell$,  by hypothesis
 there exists a constant 
$C'<1$ such that
\[
\frac{\abs{H_{X/K,\ell} - \anisotropic_{\ell,m}(G)}}{\abs{H_{X/K,\ell}}} < C'.
\]
As in Lemma \ref{lemdensityanisoone}, this implies that the set $\anisotropic_m(X/K;
G)$ has density one.

Suppose $\idp \in \anisotropic_m(X/K;G)$ is not one of the finitely
many exceptional primes allowed by \eqref{eqfieldcond}, and choose an $\ell$ such that  
$\rho_\ell(\sigma_\idp) \in \anisotropic_{\ell,m}(G)$.  Not only is $X_\idp$ simple,
but it is absolutely simple.  Indeed, let $\idp' \in M(X/K')$ be a prime
lying over $\idp$; then $\rho_\ell(\sigma_\idp^m)$ is a power of the
mod-$\ell$ reduction of the Frobenius element of $X_{\idp'} = X_\idp
\cross_{\kappa(\idp)}\kappa(\idp')$, and $X_{\idp'}$ is simple.  Moreover,
$\kappa(\idp')$ contains the field of definition of the $n$-torsion of
$X_\idp$.  Since $X_\idp$ is simple over $\kappa(\idp')$ (Lemma
\ref{lemfrobirred}), it is absolutely simple (by \eqref{eqfieldcond}).

Since $\anisotropic_m(X/K; G) \subseteq S^*(X/K)$, the set of primes at which $X$
has absolutely simple reduction has density one.
\end{proof}

\begin{proposition}
\label{propnewdensityRpos}
Suppose $X/K$ is an absolutely simple abelian variety over a number
field.   Suppose that there exist a finite Galois extension $K'/K$, an
integer $d \ge 2$,  and
an infinite set of primes $\ELL$ such that for each $\ell \in \ELL$
there exists a representation $W_{\rat_\ell}$ of $\gal(K')$ such that
$T_\ell(X)\tensor\rat_\ell \iso W_{\rat_\ell}^{\oplus d}$ as
$\gal(K')$-module.
\begin{alphabetize}
\item For $\idp$ in a set of density at least $1/[K':K]$, there exists
  an abelian variety $Y_\idp$ over $\kappa(\idp)$ such that $X_\idp \twiddle Y_\idp^{\oplus d}$.
\item Suppose $(X/K', G/\integ[1/\Delta],\ELL)$ satisfies \eqref{eqgoursat},
  and that $(X/K,K')$ satisfies \eqref{eqfieldcond}.
  Suppose there is a constant $C<1$ such that
for each $\ell \in \ELL$
and each group $H_\ell$ of type $G(\integ/\ell)$, 
$\abs{\isotropic_\ell(H_\ell)}/\abs{H_\ell} < C$.  Suppose there
exists $\sigma \in \gal(K')$ such that
$\rho_{X/K',\integ_\ell}(\sigma)$ acts irreducibly and semisimply on
$W_{\rat_\ell}$. 
Then for $\idp$ in a set of density at least $1/[K':K]$, $X_\idp\cross \overline{\kappa(\idp)}
\twiddle Y_{\bar\idp}^{\oplus d}$ for an absolutely simple abelian variety $Y_{\bar\idp}/\overline{\kappa(\idp)}$.
\end{alphabetize}
\end{proposition}

\begin{proof}
Let $T(X/K,K')$ be the set of primes $\idp \in M(X/K)$ which lie under
some $\idp' \in M(X/K')$ with prime residue field.
(Note that $T(X/K,K')$ has density at least
$1/[K':K]$.)  If $\idp \in T(X/K,K')$,  then $X_{\idp'}$ is reducible
by Lemma \ref{lemnewRall}(b).  Since $\kappa(\idp') = \kappa(\idp)$,
$X_\idp$ is reducible, too.  This proves (a).

Now supposes the hypotheses of (b) hold.  Let $T^*(X/K')$ be the set
of primes $\idp'\in M(X/K')$ such that $X_{\idp'}$ is isogenous to $Y_{\idp'}^{\oplus
  d}$ for some simple abelian variety $Y_{\idp'}/\kappa(\idp')$.  By
hypothesis \eqref{eqfieldcond}, such a $Y_{\idp'}$ is actually absolutely
simple.  By Lemma \ref{lemnewpowsimp}, $T^*(X/K')$ has density one; the
set of primes of $K$ lying under elements of $T^*(X/K')$ has density
at least $1/[K':K]$.
\end{proof}

\section{Lefschetz groups}
\label{subseclefschetz}

Suppose $X/K$ is an abelian variety whose endomorphism algebra is a
(noncommutative) division algebra.  In this section we show that the representation of
$\gal(K)$ on $T_\ell(X)$ is isomorphic to a several copies of the
same representation.  The result follows from an analogous
description of Lefschetz groups due to Milne, whose treatment
\cite{milnelefschetz} we follow here.

Consider a Weil cohomology theory $X \mapsto H^*(X)$ with coefficients
in a field $k$ of characteristic zero.  Examples of such a theory
include Betti cohomology 
(for varieties over $\cx$) and  $\ell$-adic cohomology.  If $X$ is an
abelian variety, let $V(X)_k$ be the dual 
of its first cohomology group in this cohomology theory.  For example,
$V(X)_{\rat_\ell} = T_\ell(X) \tensor_\integ\rat$; and if $X$ is a
complex abelian variety, then $V(X)_\rat$ is its first Betti homology
$H_1(X(\cx),\rat)$.  

In this context there is a Lefschetz group $\lef(X)_k$, an algebraic
group over $k$ which is naturally a subgroup of $\gl(V(X)_k)$.
It is the largest subgroup which fixes the (suitably
Tate twisted) cohomology classes of cycles on  powers of $X$ which are
linear combinations of intersections of divisor classes.

Suppose $X$ is a simple abelian variety; recall the conventions
surrounding the endomorphism algebra $D(X) = \End(X)\tensor\rat$
introduced in Section \ref{subsecabvar}.
Say that $k$ totally splits $D(X)$ if $E(X) \tensor_\rat k \iso
\oplus_{\tau:E(X) \inject k} k \iso k^{\oplus{e(x)}}$, and if for each
$\tau:E(X) \inject k$, one has $D(X)\tensor_{E(X),\tau}k \iso \mat_{d(X)}(k)$.

\begin{lemma}
\label{lemlefrepsplits}
Let $X$ be an absolutely simple abelian variety.  Consider a Weil cohomology theory with
coefficients in a field $k$, and suppose that $k$ totally splits $D(X)$.
There is a representation $W_k$ of $\lef(X)_k$ such that $V(X)_k \iso
W_k^{\oplus d(X)}$ as $\lef(X)_k$-representations.
\end{lemma}

\begin{proof}
Suppose $k$ is algebraically closed.  There exists an algebraic
group $\til \lef(X)_k$ and a natural isomorphism $\iota: \lef(X)_k \iso
\oplus_{\sigma: F(X) \inject k} \til \lef(X)_k$ \cite[Sec.\
2]{milnelefschetz}.  Moreover, there exists a representation $\til
V_k$ of $\til \lef(X)_k$ such that, under the isomorphism $\iota$,
$V(X)_k$ and $\oplus_{\sigma:F(X)\inject k} \til V_k^{\oplus d(X)}$ are
isomorphic representations of $\lef(X)_k$.  Then $W_k :=
\oplus_{\sigma: F(X)\inject k} \til V_k$ is the sought-for
decomposition of $V(X)_k$ as $\lef(X)_k$ representation.  The
analysis in \cite[Sec.\ 2]{milnelefschetz} relies only on the fact that
the field of coefficients totally splits the endomorphism algebra, and
thus the result holds under this weaker hypothesis on $k$.
\end{proof}

Recall that for an abelian variety $X$ over a number field $K$, $H_{X/K;\rat_\ell}$ is the Zariski closure of $\rho_{X/K,\integ_\ell}(\gal(K))$ in
$\gl(V(X)_{\rat_\ell})$. 

\begin{lemma}
\label{lemgalrepsplits}
Let $X/K$ be an absolutely simple abelian variety over a number field.  Suppose
$\rat_\ell$  totally splits $D(X)$. Suppose that $H_{X/K;\rat_\ell}$ is
connected.  There is a representation
$W_{\rat_\ell}$ of $\gal(K)$ such that, as $\gal(K)$-modules,
\[
V(X)_{\rat_\ell} \iso W_{\rat_\ell}^{\oplus d(X)}.
\]
\end{lemma}

\begin{proof}
  Fix an embedding $K \inject \cx$, so that $X$ has a natural
  structure of complex abelian variety.  Associated to $X$ is its
  Mumford-Tate group $\mt(X)$.  It is an algebraic subgroup of
  $\gl(V(X)_\rat)$, and there is a natural inclusion $\mt(X) \subseteq
  \lef(X)_\rat$.  Since comparison isomorphisms in cohomology furnish
  isomorphisms of Lefschetz groups, there are thus natural inclusions
  $\mt(X) \cross \rat_\ell \subseteq \lef(X)_\rat \cross \rat_\ell
  \iso \lef(X)_{\rat_\ell}$.  Work of Deligne, Piateskii-Shapiro and 
  Borovoi (see, for example \cite[Prop.\ 2.9 and Thm.\
  2.11]{delignemilne}) shows there is a natural inclusion
\begin{equation}
\label{eqmt}
H_{X/K;\rat_\ell} \subseteq \mt(X)\cross_\rat \rat_\ell.
\end{equation}
(In general, \eqref{eqmt} holds only for the connected component
$H_{X/K;\rat_\ell}^0$; the Mumford-Tate conjecture asserts that \eqref{eqmt} is actually an
equality.)   By Lemma \ref{lemlefrepsplits}, there exists a representation
$W_{\rat_\ell} \subseteq V(X)_{\rat_\ell}$ such that $V(X)_{\rat_\ell}
\iso W_{\rat_\ell}^{\oplus d(X)}$.  Therefore, $V(X)_{\rat_\ell} \iso
W_{\rat_\ell}^{\oplus d(X)}$ as $H_{X/K;\rat_\ell}$-modules, and thus
as $\gal(K)$-modules.
\end{proof}

\section{Commutative endomorphism ring}
\label{subseccommutative}

\begin{theorem}
\label{threalaction}
Let $X/K$ be an absolutely simple abelian variety over a number field.
Suppose $F = \End_{\bar K}(X_{\bar K})
\tensor\rat$ is a totally real field.
If $r = \dim X/[F:\rat]$ is odd then $S^*(X/K)$, the set of primes
where $X$ has good, absolutely simple reduction,  has density one.
\end{theorem}

\begin{proof}
Using Lemma \ref{lemgoodfield}, choose a finite Galois extension
$K'/K$ such that $(X/K,K')$ satisfies \eqref{eqfieldcond}.
Let $G =
\res_{\calo_F/\integ}\gsp_{2r}$, with derived group $G^\der =
\res_{\calo_F/\integ}\sp_{2r}$.  For all $\ell \gg 0$, the derived
group of $H_{X/K',\ell}$ is $G^\der(\integ/\ell)$ \cite[Thm.\
B]{banaszaketal06} (see also \cite{ribet76} for the case $r=1$), so that $H_{X/K',\ell}$ is of type
$G(\integ/\ell)$.  Moreover, $X_\ell$ is the natural representation of
$H_{X/K',\ell}$.
By Lemmas \ref{lemcexists}
and \ref{lemgoursat} and Proposition  \ref{propdensitySstarone},
$S^*(X/K) = 1$. 
\end{proof}

\begin{theorem}
\label{thgsp}
Let $X/K$ be an absolutely simple abelian variety over a number
field.  Suppose there is some prime $\ell_0$ such that
$H_{X/K,\rat_{\ell_0}} = \gsp_{2g}(\rat_{\ell_0})$.  Then
$S^*(X/K)$ has density one.
\end{theorem}

\begin{proof}
There is always an {\em a priori} inclusion
$H_{X/K,\rat_\ell}\subseteq \gsp_{2g}(\rat_\ell)$.
Since if the Mumford-Tate conjecture is true for $X$ at one prime
$\ell_0$ it is true at all primes \cite[Thm.\ 4.3]{larsenpink95}, the hypothesis holds for every
rational prime $\ell$.   A theorem of Larsen \cite[Thm.\
3.17]{larsenmax}, combined with Lemma \ref{lemgptheory}, implies that
for $\ell$ in a set of primes $\ELL$ of density one, the derived
subgroup $H_{X/K,\ell}^\der$ of the image of $\rho_{X/K,\ell}$
is $\sp_{2g}(\integ/\ell)$.  Choose an extension $K'/K$ such that
$(X/K,K')$ satisfies \eqref{eqfieldcond}.  Then $H_{X/K',\ell}^\der \iso \sp_{2g}(\integ/\ell)$
for $\ell$ in a set of primes $\ELL'\subseteq \ELL$ which still has
density one.  Again, Lemmas \ref{lemcexists} and 
\ref{lemgoursat} and Proposition \ref{propdensitySstarone} show that $S^*(X/K)$ has
density one. 
\end{proof}

\begin{remark}
Under the hypotheses of Theorem \ref{thgsp}, Chai and Oort show that
$S(X/K)$ has positive density \cite[Rem.\ 5.4]{chaioort01}.  Under the
apparently stronger hypothesis that $H_{X/K,\ell} =
\gsp_{2g}(\integ/\ell)$ for all $\ell\gg 0$, Chavdarov shows that
$S^*(X/K)$ has density one \cite[Cor.\ 6.10]{chavdarov}.
\end{remark}

Let $X/K$ be an absolutely simple abelian variety of dimension $g$
such that $D(X) = \End_{\bar K}(X)\tensor\rat \iso E$,
a totally imaginary extension of $\rat$.  Let $r = 2g/[E:\rat]$.  Fix
an embedding $E \inject \cx$; the tangent space $\Lie(X_\cx)$ of $X_\cx$ is a
$g$-dimensional vector space over $\cx$, and thus a module over
$E\tensor_\rat \cx \iso \oplus_{\tau:E\inject \cx} \cx$. Let
$m_\tau$ be the $\cx$-dimension of the subspace of $\Lie(X_\cx)$ on
which $E$ acts via $\tau$.  For $\tau \in \hom(E,\cx)$, let
$\bar\tau$ denote the composition of $\tau$ with complex
conjugation.  Then $m_\tau + m_{\bar\tau} = r$ is independent of the
choice of $\tau$.

Vasiu has proved the Mumford-Tate conjecture for $X$,
provided that the action of $E$ on $X$ is non-special \cite[Thm.\
1.3.4]{vasiumt}. We defer 
a full exposition to {\em loc. cit.}, but note that each of the following is an example
of a non-special action \cite[6.2.4]{vasiumt}:
\begin{romanize}
\item $r = 4$ or $r$ is prime;
\item there exists a $\tau \in \hom(E,\cx)$ such that $m_\tau = 1$;
\item  there exist $\tau$ and $\tau'$ such that $1 \le m_\tau <
  m_{\tau'} \le r/2$ and either
  $\gcd(m_\tau,r)$ or $\gcd(m_{\tau'},r)$ is $1$;
\item there exists a $\tau$ such that $\gcd(m_\tau, m_{\bar\tau}) =
  1$, and the natural numbers $(m_\tau,  m_{\bar\tau})$ are not of the
  form $(\binom{i}{j-1}, \binom{i}{j})$ for any natural numbers $i$
  and $j$.
\end{romanize}

The case of the Mumford-Tate conjecture where $[E:\rat] = 2$ and $r = g$ is prime is due to Chi
\cite[Cor.\ 3.2]{chiIV}.

\begin{theorem}
\label{thimagaction}
Let $X/K$ be an absolutely simple abelian variety over a number
field.  Suppose $E := \End_{\bar K}(X_{\bar K})\tensor\rat$ is a
totally imaginary field, and that $X$ is of non-special type.  Then
$S^*(X/K)$ has density one.
\end{theorem}

\begin{proof}
Since Murty and Patankar have proved this result for abelian varieties
of CM type \cite[Thm.\ 3.1]{murtypatankar}, we assume that $2 \dim
X/[E:\rat] > 1$.  Let $K'/K$ be a finite extension such that
$(X/K,K')$ satisfies \eqref{eqfieldcond}.  By \cite[Thm.\ 1.3.4]{vasiumt}, the Mumford-Tate conjecture is
true for each representation $\rho_{X/K,\rat_\ell}\rest{\gal(K')}$.  More
precisely, there is a group $G/\integ[1/\Delta]$ of type
\eqref{caseimag} such that for almost all $\ell$, the Zariski closure
of $H_{X/K',\rat_\ell}$ is isomorphic to a subgroup of $G(\rat_\ell)$
which contains $G^\der(\rat_\ell)$.  By
\cite[Thm.\ 3.17]{larsenmax} and Lemma \ref{lemgptheory}, there is a
set of primes $\ELL$ of density one such 
that for $\ell \in \ELL$, $H_{X/K',\ell}$ is of type
$G(\integ/\ell)$.  Moreover, $X_\ell$ is the natural representation of
$H_{X/K',\ell}$.  Since the groups $G(\integ/\ell)$ satisfy Goursat's 
lemma (Lemma \ref{lemgoursat}), $S^*(X/K)$ has density one by
Lemma \ref{lemcexists} and Proposition \ref{propdensitySstarone}.
\end{proof}

Via the Torelli functor, these results yield information about curves.  For example,
consider the following condition on a curve $C$ over a field $k$:

\begin{equation}
\label{eqcurvecond}
\text{If $C \ra D$ is finite of degree at least $2$, then $D$ has genus zero.}
\end{equation}

If the Jacobian $\jac(C)$ is simple, then $C$ satisfies
\eqref{eqcurvecond}.  The converse is true if the genus of $C$ is at most
$6$, since almost every principally polarized abelian variety is a
Jacobian in dimension at most $3$.

\begin{corollary}
\label{corcurve}
Let $C/K$ be a curve of genus of odd prime genus $g$ over a number
field such that $C/\bar K$ satisfies
\eqref{eqcurvecond}.  Suppose that either $g\in \st{3,5}$ or that
$\jac(C)$ is absolutely simple. For almost all primes $\idp$,
$C_\idp/\kappa(\idp)$ satisfies \eqref{eqcurvecond}.
\end{corollary}

\begin{proof}
By hypothesis (and the preceding discussion), $\jac(C)$ is absolutely
simple.  A simple
abelian variety of odd prime dimension over a number field has
commutative endomorphism ring.  This endomorphism ring is totally real
or totally imaginary; and in the latter case, the action is not
special.  Now use  Theorem \ref{threalaction} or
\ref{thimagaction} as appropriate.
\end{proof}

In general, a curve $C$ which satisfies \eqref{eqcurvecond} need not have
reductions $C_\idp$ satisfying \eqref{eqcurvecond} for a dense, or even
infinite, set of
primes $\idp$.  Indeed, let $C$ be the second curve considered
in the introduction.  Then $\jac(C)$ is simple, thus $C$ satisfies
\eqref{eqcurvecond}; but for each prime $\idp$ of good reduction, $\jac(C_\idp
\cross \overline{\kappa(\idp)})$ dominates, and thus $C_\idp$ covers, an
elliptic curve.

\section{Noncommutative endomorphism ring}

Recall the definitions of $D(X)$ and $d(X)$ from Section \ref{subsecabvar}.

\begin{proposition}
\label{propnewnoncommutative}
Let $X/K$ be an absolutely simple abelian variety over a number
field.  Suppose that $\End_{\bar K}(X_{\bar K})$ is
noncommutative.
\begin{alphabetize}
\item Then $R(X/K)$, the set of primes $\idp$ such that $X_\idp$ is
  reducible, has positive density.
\item If $\End_K(X) = \End_{\bar K}(X)$ and $H_{X/K,\rat_\ell}$ is
  connected, then $R(X/K)$ has density one.
\end{alphabetize}
\end{proposition}

\begin{proof}
Let $K'/K$ be a finite extension such that $(X/K,K')$ satisfies
\eqref{eqfieldcond}; such an extension exists by Lemma
\ref{lemgoodfield}.  The conclusion of (b) for $X_{K'}$ implies the
conclusion of (a) for $X$.  Therefore, it suffices to assume
$\End_K(X) = \End_{\bar K}(X)$ and that $H_{X/K,\rat_\ell}$ is
connected, and then prove
that $R(X/K) = M(X/K)$.

Consider the set $\ELL$ of
primes $\ell$ such that $\rat_\ell$ totally splits $D(X)$.  Note that $\ELL$
has positive density, and in particular is infinite.  Suppose $\ell
\in \ELL$.  By Lemma \ref{lemgalrepsplits}, there exists a
representation $W_{\rat_\ell}$ of $\gal(K)$ such that $T_\ell(X)
\tensor\rat \iso W_{\rat_\ell}^{\oplus d(X)}$ as $\gal(K)$-module.
Note that $d(X) > 1$ since $D(X)$ is noncommutative.  By Lemma
\ref{lemnewRall}, for $\idp$ in a subset of $M(X/K)$ of density one,
$X_\idp$ is isogenous to $Y_\idp^{\oplus d(X)}$ for some abelian 
variety $Y_\idp/\kappa(\idp)$.
\end{proof}

\begin{remark}
  In the special case of an abelian surface $X$ with action by an indefinite
  quaternion algebra, $X$ has absolutely split reduction at every
  prime of good reduction.  This is explained in detail in \cite[Sec.\
  2]{murtypatankar05}; see also \cite[Rem.\ 5.(ii)]{chaioort01}.  It
  is possible that Lemma \ref{lemcenterramifies} will yield a generalization of
  this. 
In   the context of Proposition \ref{propnewnoncommutative}, Murty and Patankar
  show \cite[Prop.\ 5.4]{murtypatankar} that $X_\idp$ is not simple at
  any prime $\idp$ of {\em ordinary} reduction.
\end{remark}

\begin{lemma}
\label{lemcenterramifies}
Let $X/K$ be an absolutely simple abelian variety over a number
field with noncommutative endomorphism algebra $D(X)$.  Let $\Delta$
be the product of all (finite) primes of $E(X)$ which 
ramify in $D(X)$.  Suppose $\idp \in M(X/K)$ is relatively prime to
$\Delta$.  Then $E(X_\idp)$ is ramified at every prime dividing
$\Delta$.
\end{lemma}

\begin{proof}
Suppose $X_\idp$ is simple, and let $p$ be the characteristic of
$\kappa(\idp)$.  The inclusion $\End(X)\inject
\End(X_\idp)$ forces $D(X_\idp)$ to be noncommutative.  By
\cite[Thm. 2(e)]{tateendff}, $D(X_\idp)$ is split at all primes not
dividing $p$.  In particular, $D(X_\idp)$ is split at all primes not
dividing $\Delta$.  Since only a ramified field extension splits a
division algebra over a local field, $E(X_\idp)$ must ramify at all
primes dividing $\Delta$.
\end{proof}

If the endomorphism ring of $X$ is an indefinite quaternion algebra,
one knows more about the structure of the reductions $X_\idp$:

\begin{theorem}
\label{thindefquat}
Let $X/K$ be an absolutely simple abelian variety over a number field.  Suppose that 
$\End_{\bar K}(X_{\bar K}) \tensor \rat$ is an indefinite
quaternion algebra over a totally real field $F$.  If $\dim X/2[F:\rat]$ is odd,
then for $\idp$ in a set of positive density, $X_\idp$ is geometrically isogenous to the self-product of an
absolutely simple abelian variety $Y_{\bar\idp}/\overline{\kappa(\idp)}$ of
dimension $(\dim X)/2$.
\end{theorem}

\begin{proof}
Let $K'/K$ be a finite extension of $K$ such that $(X/K,K')$ satisfies
\eqref{eqfieldcond}. 
By \cite[Thm.\ B]{banaszaketal06}, 
Lemma \ref{lemcexists} and Lemma \ref{lemgoursat}, there exists a
group $G/\integ[1/\Delta]$ of type \eqref{casereal} and an infinite
set of primes $\ELL$ such that $(X/K', G/\integ[1/\Delta], \ELL)$
satisfies \eqref{eqgoursat}.  Moreover, there is a $G$-module
$W/\integ[1/\Delta]$  such that for all $\ell \in \ELL$, $W\tensor
\integ/\ell$ is an irreducible $G(\integ/\ell)$-module and $X_\ell
\iso (W\oplus W)\tensor \integ/\ell$ as $G(\integ/\ell)$-module.
(This is \cite[Thm.\ 5.4]{banaszaketal06}; see also \cite{chiIIsinica}
for the analogous statment for $\rat_\ell$-modules.)  The result now
follows from Proposition \ref{propnewdensityRpos}.
\end{proof}

\section*{Acknowledgements}

My debt to the work of Murty and Patankar is obvious, and I thank
V.\ Kumar Murty for sharing \cite{murtypatankar} with me in preprint
form.   It's also a pleasure to acknowledge helpful discussions with Bob Guralnick, Tim Penttila, Rachel Pries and David
Zywina, as well as comments from the referee.

\bibliographystyle{hamsplain}
\bibliography{jda}

\end{document}